\documentclass[12pt]{amsart}
\usepackage[utf8]{inputenc}
\usepackage{amsmath}
\usepackage{amsthm}
\usepackage{amsfonts}
\usepackage{amssymb}
\usepackage{tabu}
\usepackage{latexsym} 
\usepackage[margin=1in]{geometry}
\usepackage{enumerate} 
\usepackage[colorlinks=true]{hyperref}
\usepackage{mathrsfs}
\usepackage{tikz}
\usepackage[parfill]{parskip}
\usetikzlibrary{matrix,arrows,decorations.pathmorphing}
\usepackage{tikz-cd}
\usepackage[all]{xy}
\usepackage{setspace}
\usepackage{tabularx}
\usepackage{xcolor}
\usepackage{framed}

\newcommand{\Ext}{\mathrm{Ext}}

\newcommand{\Hom}{\mathrm{Hom}}

\newcommand{\m}{\mathfrak{m}}

\newcommand{\module}{\mathrm{mod}}

\newtheorem{theorem}{Theorem}[section]
\newtheorem*{theorem*}{Theorem}
\newtheorem{lemma}[theorem]{Lemma}
\newtheorem*{lemma*}{Lemma}

\newtheorem*{proposition*}{Proposition}

\newtheorem*{corollary*}{Corollary}

\theoremstyle{definition}
\newtheorem{definition}[theorem]{Definition}
\newtheorem*{definition*}{Definition}

\newtheorem{remark}[theorem]{Remark}

\newtheorem{introtheorem}{Theorem}

\numberwithin{equation}{theorem}
\newcommand{\Tr}{\mathrm{Tr}}

\title{Arf rings, simple singularities and reflexive modules}
\author[\"{O}.~Esentepe]{\"{O}zg\"{u}r Esentepe}
\address{Institut für Mathematik und Wissenschaftliches Rechnen \\Universität Graz \\ 
Heinrichstraße 36, 8010 Graz, Austria}
\email{ozgur.esentepe@uni-graz.at}
\urladdr{https://www.sntp.ca}


\begin{document}
\maketitle

\begin{abstract}
    In a pandemic era preprint, Dao showed showed two remarkable properties of Arf rings: under some mild conditions, they admit finitely many indecomposable reflexive modules up to isomorphism and every reflexive module is actually isomorphic to its own dual. In fact, the latter property characterises Arf rings. Arf rings are one dimensional rings and it is natural to wonder what happens in higher Krull dimension. In this paper, we investigate the self-dual property for commutative Noetherian local rings.
\end{abstract}

\section{Introduction}

Arf rings were introduced by Joseph Lipman in \cite{Lipman} following the work of Cahit Arf \cite{Arf}. They play an essential role in the study of curve singularities and Lipman utilized the theory of Arf rings to prove, under some mild conditions, a conjecture of Zariski related to strict closures. Arf rings regained popularity amongst commutative algebraists in the last decade, after half a century. In \cite{celikbas-celikbas-ciuperca-endo-goto-isobe} authors proved Zariski's conjecture in full generality. In \cite{Isobe-Kumashiro,Dao,Dao-Lindo}, authors independently showed that Arf rings are of \textit{finite reflexive type}. That is, there are finitely many indecomposable reflexive modules over an Arf ring up to isomorphism, and up to some reasonable assumptions which make the theory work. Dao also showed in \cite{Dao} that if $R$ is a semi-local Cohen-Macaulay equidimensional ring of dimension one whose completion (with respect to its Jacobson radical) is reduced, then $R$ is Arf if and only if every reflexive $R$-module is \textit{self-dual}, that is, isomorphic to its own dual.

This philosophy of relating properties of subcategories of the module category of a local ring with the properties of the corresponding geometric object has emerged in the 1980s. As an important example Auslander \cite{Auslander} and Huneke-Leuschke \cite{Huneke-Leuschke} proved that a Cohen-Macaulay local ring of finite Cohen-Macaulay type is an isolated singularity. In the Gorenstein case, we have an even better picture. A complete Gorenstein ring of finite Cohen-Macaulay type is a simple singularity coming from Arnol'd's work on germs of holomorphic functions \cite{Arnold}. This was proved by the work of Buchweitz-Greuel-Schreyer \cite{Buc-Greu-Schr} and \cite{Knorr}. Following the same philosophy, Christensen-Piepmeyer-Striuli-Takahashi showed two decades later that if $R$ is a complete commutative Noetherian local ring such that $R$ admits finitely many indecomposable Gorenstein projective modules up to isomorphism (and if it admits at least one non-free such module), then it has to be a simple singularity \cite{Christensen-Piepmeyer-Striuli-Takahashi}.

Motivated by all this previous work, we consider two conditions on a commutative Noetherian ring $R$:
\begin{description}
    \item[frt] there are finitely many indecomposable reflexive modules up to isomorphism,
    \item[sd] every reflexive modules is self-dual.  
\end{description}
It is an active research area to study rings which have the property \textbf{(frt)}. In this paper, we focus on the property \textbf{(sd)}. Luckily for us, in lower dimensions we already have satisfactory answers:
\begin{enumerate}
    \item As mentioned above, in dimension 1, a complete Cohen-Macaulay equidimensional ring with a reduced completion has the property \textbf{(sd)} if and only if it is Arf \cite[Theorem B]{Dao}.
    \item In dimension 2, we have found an answer in \cite[Theorem 2.7]{Ooishi} by Ooishi. Let $R$ be a two-dimensional complex analytic normal local domain which is not a regular local ring. Then, $R$ satisfies \textbf{(sd)} if and only if $R$ is a simple singularity given by one of the following polynomials:
    \begin{description}
        \item[($A_1$)] $x^2 + y^2 + z^2$,
        \item[($D_n$)] $x^2y + y^{n-1} + z^2$ with $n$ even and $n \geq 4$,
        \item[($E_7$)] $x^3 + xy^3 + z^2$, 
        \item[($E_8$)] $x^3 + y^5 + z^2$  .
    \end{description}
\end{enumerate}
We prove the following two theorems which illustrate the homological similarity (and unlikeness) of Arf rings to simple singularities. The first main theorem shows that the self-dual property can only be satisfied by hypersurface rings when the depth is at least  two.
\begin{introtheorem}
    Let $R$ be a commutative Noetherian local ring of depth $t \geq 2$. If $R$ satisfies the condition \textbf{(sd)}, then it is a hypersurface ring of dimension at most three.
\end{introtheorem}
The second main theorem also covers rings which have smaller depth than two.
\begin{introtheorem}
    Let $R$ be a complete commutative Noetherian local ring. If $R$ satisfies the conditions \textbf{(frt)} and \textbf{(sd)}, then it is a simple singularity of dimension at most three provided that it admits a non-free 3-torsionfree module.
\end{introtheorem}

The proofs of these theorems depend on a simple observation on $3$-torsionfree modules over a ring that satisfies \textbf{(sd)}. In the next section, we will recall the theory of $n$-torsionfree modules and prove our theorems.

\section{Self-dual modules}

We will start with giving the necessary background, set notations and conventions.Throughout this paper, we will always work with commutative Noetherian rings and all modules will be assumed to be finitely generated. We will denote by $\module R$ the category of finitely generated $R$-modules.

\begin{definition}
    Let $R$ be a commutative Noetherian ring and $M$ be a finitely generated $R$-module.
    \begin{enumerate}
        \item We denote by $(-)^*$ the $R$-dual functor $\Hom_R(-,R) : \module R \to \module R$. We call $M^*$ the \textit{dual} of $M$. We say that $M$ is \textit{reflexive} if the canonical map $M \to M^{**}$ is an isomorphism and $M$ is \textit{self-dual} if there is an isomorphism $M \cong M^*$. 
        \item Given a projective $R$-module $P$ and a surjective morphism $P \to M \to 0$, we call the kernel a \textit{syzygy} of $M$ and denote it by $\Omega_R M$. This is uniquely determined up to projective summands by Schanuel's Lemma. We drop the subscript when the context is clear.
        \item Given a projective presentation $P_1 \xrightarrow{d} P_0 \to M \to 0$ of $M$, we call the cokernel of $d^*$ an \textit{Auslander transpose} of $M$ and denote it by $\Tr M$. It is also uniquely determined up to projective summands \cite{Auslander-Bridger}.
        \item For a positive integer $n$, we say that $M$ is \textit{$n$-torsionfree} if $\Ext_R^i(\Tr M, R) = 0$ for $i = 1, \ldots, n$.
        \item We say that $M$ is a \textit{Gorenstein projective} or \textit{totally reflexive} if $M$ is reflexive and for all $i > 0$, we have $\Ext_R^i(M, R) = \Ext_R^i(M^*, R) = 0$. 
\end{enumerate}
\end{definition}

We recall some well-known facts.
\begin{remark}
    Let $R$ be a commutative Noetherian ring and $M$ be a finitely generated $R$-module.
    \begin{enumerate}
        \item The module $M$ is reflexive if and only if there is \textit{some} isomorphism $M \cong M^{**}$ \cite[Proposition 1.1.9]{Christensen}. Therefore, a self-dual $R$-module is reflexive.
        \item If $R$ is a local normal domain of Krull dimension 2, then a module $M$ is reflexive if and only if it is maximal Cohen-Macaulay. See, for instance, \cite[A.14 Corollary]{Leuschke-Wiegand}.
        \item We have an exact sequence 
        \begin{align*}
        0 \to \Ext_R^1(\Tr M, R) \to M \to M^{**} \to \Ext_R^2(\Tr M, R) \to 0
        \end{align*}
        which shows that $M$ is reflexive if and only if it is $2$-torsionfree. This was proved by Auslander-Bridger in \cite{Auslander-Bridger}.
        \item If $M$ is $n$-torsionfree and $\Ext_R^1(M,R) = 0$, then $\Omega M$ is $(n+1)$-torsionfree. See, for instance, \cite[Lemma 2.3]{Iyama-Wemyss} and the preceding discussion.       
    \end{enumerate}    
\end{remark}

Let us start with a simple observation.

\begin{lemma}\label{first-lemma}
    Let $R$ be a commutative Noetherian ring and $M$ be a finitely generated $R$-module. If $M$ is self-dual and $3$-torsionfree, then we have $\Ext_R^{1}(M,R) = 0$. 
\end{lemma}
\begin{proof}
    Let $M$ be self-dual and $3$-torsionfree. Then, we have 
    \begin{align*}
     \Ext_R^1(M,R) = \Ext_R^1(M^*, R)= \Ext_R^1(\Omega^2 \Tr M, R) = \Ext_R^3(\Tr M, R)=0.
    \end{align*}  
\end{proof}
The following lemma follows from Lemma \ref{first-lemma}. It shows that if all reflexive modules are self-dual, then $3$-torsionfree modules and totally reflexive modules coincide.
\begin{lemma}\label{second-lemma}
    Let $R$ be a commutative Noetherian ring that satisfies the property \textbf{(sd)}. Assume $M$ is a finitely generated $R$-module.
    \begin{enumerate}
        \item $M$ is $3$-torsionfree if and only if it is totally reflexive.
        \item If for some integer $n \geq 0$, the $n$th syzygy $\Omega^n M$ is $3$-torsionfree, then for any $i > n$, we have $\Ext^i_R(M,R) = 0$.
    \end{enumerate}
\end{lemma}
\begin{proof}
    We start with noting that a self-dual module $M$ is totally reflexive if and only if $\Ext_R^i(M,R) = 0$ for all $i > 0$. A totally reflexive module is always $3$-torsionfree. We will show the converse. Assume that $M$ is $3$-torsionfree. In particular, it is reflexive and by our assumption it is self-dual. Hence Lemma \ref{first-lemma} implies that $\Ext_R^1(M,R) = 0$. This implies that $\Omega M$ is $4$-torsionfree and in particular $3$-torsionfree and we can apply Lemma \ref{first-lemma} to get $\Ext_R^2(M,R) = \Ext_R^1(\Omega M, R) = 0$. By induction, this proves the first assertion.

    The second assertion follows from the first as we have $\Ext_R^j(\Omega^n N, R) = \Ext^{j+n}_R(M,R)$ for all $j > 0$.
\end{proof}
\begin{remark}
    Let $(R, \m,k)$ be a commutative Noetherian local ring that satisfies the property \textbf{(sd)}. If the depth $t$ of $R$ is at least 2, then $R$ is Gorenstein. This is an immediate consequence of Lemma \ref{second-lemma}. Indeed, by \cite[Theorem 4.1]{Dey-Takahashi}, we know that $\Omega^t k$ is $3$-torsionfree. Therefore, it is totally reflexive. By definition, this means $k$ has finite Gorenstein dimension. This implies that $R$ is Gorenstein \cite{Auslander-Bridger}. However, our main theorem improves this remark as we prove below.
\end{remark}
The self-dual property emposes a bound on the depth as we prove below.
\begin{lemma}\label{depth-at-most-three}
    Let $(R, \m,k)$ be a commutative Noetherian local ring that satisfies the property \textbf{(sd)}. Then, the depth of $R$ is at most three.
\end{lemma}
\begin{proof}
    Assume that the depth of $R$ is $t \geq 4$. Then, we have $\Ext_R^i(k, R) = 0$ for $i = 1, 2, 3$ and $\Ext_R^t(k,R) \neq 0$. The vanishing implies that $\Omega^3 k$ is $3$-torsionfree. Then, by Lemma \ref{second-lemma}(2), we see that $\Ext_R^i(k,R) = 0$ for all $i > 3$ which contradicts with the nonvanishing $\Ext_R^t(k,R) \neq 0$.
\end{proof}
The self-dual property also has control over the structure of free resolutions.
\begin{lemma}\label{keylemma}
    Let $R$ be a commutative Noetherian ring and $M$ be a $3$-torsionfree $R$-module. If both $M$ and $\Omega M$ are self-dual, then $M$ has a $2$-periodic resolution.
\end{lemma}
\begin{proof}
    Since $M$ is self-dual and $3$-torsionfree, we have the vanishing $\Ext_R^1(M,R) = 0$ by Lemma \ref{first-lemma}. We start with a short exact sequence
    \begin{align*}
        0 \to \Omega M \to F \to M \to 0
    \end{align*}
    defining $\Omega M$ and apply $(-)^*$ to it. This gives us an exact sequence 
    \begin{align*}
    0 \to M^* \to F^* \to (\Omega M)^* \to \Ext_R^1(M,R)
    \end{align*}
    which, in turn, gives us a short exact sequence 
    \begin{align*}
    0 \to M \to F \to \Omega M \to 0
    \end{align*}
    by our assumptions. By splicing these short exact sequences, we get a 2-periodic free resolution 
    \begin{align*}
        \xymatrix{
         &&  M\ar@{-->}[dr] &&  &&  \\
         \ldots \ar[r]&F\ar[ur] \ar[rr]&  &F \ar@{-->}[dr]\ar[rr]&  &F\ar[r]& M \to 0 \\
        \Omega M \ar[ur]&&  && \Omega M\ar@{.>}[ur] &&  
        }
    \end{align*}
    of $M$, as required.
\end{proof}
Now we are able to prove our first theorem.
\begin{theorem}
    Let $R$ be a commutative Noetherian local ring of depth $t \geq 2$. If $R$ satisfies the condition \textbf{(sd)}, then it is a hypersurface ring of dimension at most three.
\end{theorem}
\begin{proof}
    Since $R$ has depth $t$  at least $2$, we know by \cite[Theorem 4.1]{Dey-Takahashi} that the $t$-th syzygy $\Omega^t k$ is $3$-torsionfree. Since it is self-dual by Lemma \ref{first-lemma}, we know that $\Omega^{t+1}k$ is also $3$-torsionfree. Now, we have two consecutive syzygies which are self-dual and 3-torsionfree. Therefore, Lemma \ref{keylemma} tells us that the ground field $k$ must have an eventually periodic resolution. In particular, the Betti numbers of $k$ are eventually constant which can only happen if $R$ is a hypersurface \cite[Remark 8.1.1.3]{Avramov}. The last assertion follows from Lemma \ref{depth-at-most-three} as the depth of a hypersurface ring equals the Krull dimension.
\end{proof}
If we further assume the \textbf{(frt)} property, we can remove the condition on the depth.
\begin{theorem}
    Let $R$ be a complete commutative Noetherian local ring. If $R$ satisfies the conditions \textbf{(frt)} and \textbf{(sd)}, then it is a simple singularity of dimension at most three provided that it admits a non-free $3$-torsionfree module. 
\end{theorem}
\begin{proof}
    If there are finitely many reflexive modules over $R$, then there are finitely many totally reflexive modules over $R$. With the condition \textbf{(sd)}, existence of a non-free $3$-torsionfree module is equivalent to the existence of a non-free totally reflexive module. Then, by \cite[Theorem A]{Christensen-Piepmeyer-Striuli-Takahashi} we conclude that $R$ has to be a simple singularity.
\end{proof}
We finish by a remark on $3$-torsionfree modules over Arf rings.
\begin{remark}
    Lipman, in his original work on Arf rings \cite{Lipman}, proved that a local Arf ring must have minimal multiplicity. On the other hand, Avramov and Martinskovsky proved in \cite[Example 3.5]{Avramov-Martsinkovsky} that a totally-reflexive module over a Golod local ring which is not a hypersurface must be free. Therefore, we can conclude that non-Gorenstein Arf local rings can not admit 3-torsionfree modules. That is, if $R$ is a non-Gorenstein Arf local ring and $M$ is a finitely generated non-free reflexive $R$-module, then $\Ext_R^1(M,R) \neq 0$. 
\end{remark}
\section*{Acknowledgments}

I would like to thank the members of the Graz Algebra and Geometry Seminar for listening to an earlier version of this paper and offering insights. Also many thanks to Hailong Dao for helpful comments.

\bibliographystyle{alpha}
\bibliography{arf-arxiv.bib}

\newcommand{\etalchar}[1]{$^{#1}$}
\begin{thebibliography}{CCC{\etalchar{+}}23}

\bibitem[AB69]{Auslander-Bridger}
M.~Auslander and M.~Bridger.
\newblock {\em Stable module theory}, volume~94 of {\em Mem. Am. Math. Soc.}
\newblock Providence, RI: American Mathematical Society (AMS), 1969.

\bibitem[AM02]{Avramov-Martsinkovsky}
Luchezar~L. Avramov and Alex Martsinkovsky.
\newblock Absolute, relative, and {Tate} cohomology of modules of finite
  {Gorenstein} dimension.
\newblock {\em Proc. Lond. Math. Soc. (3)}, 85(2):393--440, 2002.

\bibitem[Arf48]{Arf}
Cahit Arf.
\newblock Une interpretation alg{\'e}brique de la suite des ordres de
  multiplicit{\'e} d'une branche alg{\'e}brique.
\newblock {\em Proc. Lond. Math. Soc. (2)}, 50:256--287, 1948.

\bibitem[Arn75]{Arnold}
V.~I. Arnold.
\newblock Critical points of smooth functions.
\newblock Proc. int. {Congr}. {Math}., {Vancouver} 1974, {Vol}. 1, 19-39
  (1975)., 1975.

\bibitem[Aus85]{Auslander}
M.~Auslander.
\newblock Finite type implies isolated singularity.
\newblock Orders and their applications, {Proc}. {Conf}., {Oberwolfach}/{Ger}.
  1984, {Lect}. {Notes} {Math}. 1142, 1-4 (1985)., 1985.

\bibitem[Avr98]{Avramov}
Luchezar~L. Avramov.
\newblock Infinite free resolutions.
\newblock In {\em Six lectures on commutative algebra. Lectures presented at
  the summer school, Bellaterra, Spain, July 16--26, 1996}, pages 1--118.
  Basel: Birkh{\"a}user, 1998.

\bibitem[BGS87]{Buc-Greu-Schr}
R.-O. Buchweitz, G.-M. Greuel, and F.-O. Schreyer.
\newblock Cohen-{Macaulay} modules on hypersurface singularities. {II}.
\newblock {\em Invent. Math.}, 88:165--182, 1987.

\bibitem[CCC{\etalchar{+}}23]{celikbas-celikbas-ciuperca-endo-goto-isobe}
Ela Celikbas, Olgur Celikbas, C{\u{a}}t{\u{a}}lin Ciuperc{\u{a}}, Naoki Endo,
  Shiro Goto, Ryotaro Isobe, and Naoyuki Matsuoka.
\newblock On the ubiquity of {Arf} rings.
\newblock {\em J. Commut. Algebra}, 15(2):177--231, 2023.

\bibitem[CPST08]{Christensen-Piepmeyer-Striuli-Takahashi}
Lars~Winther Christensen, Greg Piepmeyer, Janet Striuli, and Ryo Takahashi.
\newblock Finite {Gorenstein} representation type implies simple singularity.
\newblock {\em Adv. Math.}, 218(4):1012--1026, 2008.

\bibitem[Dao21]{Dao}
Hailong Dao.
\newblock Reflexive modules, self-dual modules and {Arf} rings.
\newblock Preprint, {arXiv}:2105.12240 [math.{AC}] (2021), 2021.

\bibitem[DL24]{Dao-Lindo}
Hailong Dao and Haydee Lindo.
\newblock Stable trace ideals and applications.
\newblock {\em Collect. Math.}, 75(2):395--407, 2024.

\bibitem[DT23]{Dey-Takahashi}
Souvik Dey and Ryo Takahashi.
\newblock On the subcategories of {{\(n\)}}-torsionfree modules and related
  modules.
\newblock {\em Collect. Math.}, 74(1):113--132, 2023.

\bibitem[HL02]{Huneke-Leuschke}
Craig Huneke and Graham~J. Leuschke.
\newblock Two theorems about maximal {Cohen}-{Macaulay} modules.
\newblock {\em Math. Ann.}, 324(2):391--404, 2002.

\bibitem[IK24]{Isobe-Kumashiro}
Ryotaro Isobe and Shinya Kumashiro.
\newblock Reflexive modules over {Arf} local rings.
\newblock {\em Taiwanese J. Math.}, 28(5):865--875, 2024.

\bibitem[IW10]{Iyama-Wemyss}
Osamu Iyama and Michael Wemyss.
\newblock The classification of special {Cohen}-{Macaulay} modules.
\newblock {\em Math. Z.}, 265(1):41--83, 2010.

\bibitem[Kn{\"o}87]{Knorr}
Horst Kn{\"o}rrer.
\newblock Cohen-{Macaulay} modules on hypersurface singularities. {I}.
\newblock {\em Invent. Math.}, 88:153--164, 1987.

\bibitem[Lip71]{Lipman}
Joseph Lipman.
\newblock Stable ideals and {Arf} rings.
\newblock {\em Am. J. Math.}, 93:649--685, 1971.

\bibitem[LW12]{Leuschke-Wiegand}
Graham~J. Leuschke and Roger Wiegand.
\newblock {\em Cohen-{Macaulay} representations}, volume 181 of {\em Math.
  Surv. Monogr.}
\newblock Providence, RI: American Mathematical Society (AMS), 2012.

\bibitem[Ooi96]{Ooishi}
Akira Ooishi.
\newblock On the self-dual maximal {Cohen}-{Macaulay} modules.
\newblock {\em J. Pure Appl. Algebra}, 106(1):93--102, 1996.

\bibitem[WC00]{Christensen}
Lars Winther~Christensen.
\newblock {\em Gorenstein dimensions}, volume 1747 of {\em Lect. Notes Math.}
\newblock Berlin: Springer, 2000.

\end{thebibliography}

\end{document}